\documentclass[11pt]{amsart}
\usepackage[T1]{fontenc}

\usepackage{amssymb,url,xspace}

\usepackage{mathrsfs} 

\usepackage{enumerate} 
\usepackage{amsmath} 

\numberwithin{equation}{section}

\newtheorem{assumptions}{Assumptions}[section]
\newtheorem{prop}{Proposition}[section]
\newtheorem{theo}{Theorem}[section]
\newtheorem{rema}{Remark}[section]
\newtheorem{lemm}{Lemma}[section]
\newtheorem{coro}{Corollary}[section]
\newtheorem{definition}{Definition}[section]

\newcommand{\ccd}[0]
{(\cdot)}

\newcommand{\ttnn}[1]
{\textnormal{#1}}

\newcommand{\rif}[1]
{(\ref{#1})}

\newcommand{\norm}[1]
{\left\| #1 \right\|}

\newcommand{\ra}[0]
{\rightarrow}

\newcommand{\rt}[0]
{(t)}

\newcommand{\rs}[0]
{(s)}

\newcommand{\mineq}[1]
{\leq #1}

\newcommand{\mageq}[1]
{\geq #1}

\newcommand{\equazioneref}[2]
{\begin{equation}\label{#1} \begin{split}
#2
\end{split} \end{equation}}

\newcommand{\eee}[1]
{\begin{align*} 
#1
 \end{align*}
}

\newcommand{\eps}[0]
{\varepsilon }

\newcommand{\ccal}[1]
{\mathcal{#1}}

\newcommand{\scr}[1]
{\mathscr{#1}}

\newcommand{\bb}[1]
{\mathbb{#1}}

\newcommand{\halfline}[0]
{ \bb R^+  }

\title[Control problems on infinite horizon]{Control problems on infinite horizon subject to time-dependent pure state constraints}
\date {\today}

\author{Vincenzo Basco}
\address{Thales Alenia Space. Via Saccomuro, 24 -  Rome (Italy)}
\email{vincenzo.basco@thalesaleniaspace.com}

\begin{document}

\begin{abstract}
In the last decades, control problems with infinite horizons and discount factors have become increasingly central not only for economics but also for applications in artificial intelligence and machine learning. The strong links between reinforcement learning and control theory have led to major efforts towards the development of algorithms to learn how to solve constrained control problems. In particular, discount plays a role in addressing the challenges that come with models that have unbounded disturbances. Although algorithms have been extensively explored, few results take into account time-dependent state constraints, which are imposed in most real-world control applications. For this purpose, here we investigate feasibility and sufficient conditions for Lipschitz regularity of the value function for a class of discounted infinite horizon optimal control problems subject to time-dependent constraints. We focus on problems with data that allow nonautonomous dynamics, and Lagrangian and state constraints that can be unbounded with possibly nonsmooth boundaries.
\end{abstract}

\maketitle

\tableofcontents

\section{Introduction}
Infinite time horizon models arising in mathematical economics and engineering typically involve control systems with restrictions on both controls and states. Models of optimal allocation of economic resources were, in the late 50s, among the key incentives for the creation of the mathematical theory of optimal control. Constrained optimal control problems are often solved in practical control applications, which are more challenging to deal with than unconstrained ones. 

Over the last few decades, an increasingly central role has been given to infinite horizon control problems with discount factors not only for applications in finance but for applications to artificial intelligence and machine learning. Strong connections between reinforcement learning  and control theory have prompted a major effort towards developing algorithms to learn optimal solutions. Discounting plays a role in addressing the challenges that come with models where unbounded disturbances are present. The discount configuration is common in many stochastic control problems \cite{bertsekas2012dynamic,kamgarpour2017infinite,kouvaritakis2006mpc,van2013infinite}, reinforcement learning \cite{bertsekas2019reinforcement}, and financial engineering \cite{FRANKEL2016396,nystrup2019multi}. Discount factors ensure the feasibility of constrained optimal control problems with potentially unbounded perturbations. In dynamic programming discounting is often used to ensure well-posedness of  problems with infinite horizons and possibly unlimited costs \cite{bascofrankowska2018lipschitz,blackwell1965discounted}. Moreover, with an appropriate value of the discount factor, stability is guaranteed  \cite{postoyan2016stability}.
 
Much of the present works in literature focuses on manage constraints in control problems. In general, state constraints imply non-convex feasible sets. So, there are several ways to provide amenable approximations in deterministic and probabilistic frameworks, e.g. using  available informations from probability distributions \cite{schildbach2015linear}, deterministic approximation methods jointly with confidence sets \cite{kouvaritakis2010explicit}, deterministic and stochastic tubes \cite{basco2019hamilton,cannon2010stochastic},  attainable sets \cite{hewing2018stochastic} or conservative approach with probabilistic inequalities \cite{farina2013probabilistic,hashimoto2013probabilistic} and random methods \cite{calafiore2012robust,margellos2014road}. Although deterministic algorithms have been widely investigated to solve the optimal regulation problems, few results consider the solution of optimal synthesis in the presence of time-dependent state constraints, needed for most real-world control applications (cfr. Section 2). A fundamental point in constrained cases is how to ensure desirable properties, as existence of viable solutions and stability, including regularity of the value function. 
This is particularly evidend in reinforcement learning, where value functions necessitates employing a function approximator with a limited set of parameters. Several researchers have emphasized that integrating reinforcement learning algorithms subject to state constraints with general approximation systems, such as neural networks, fuzzy sets, or polynomial approximators, can lead to unstable or divergent outcomes, even for straightforward problems (cfr. \cite{baird1995residual,boyan1994generalization,gordon1995stable}). 

%
%
%
In this settings, a key role is given by the dynamic programming principle and the Hamilton-Jacobi-Bellman (HJB) equation associated with the control problem \cite{bascofrankowska2018lipschitz,vinter00}. The value function, when differentiable, solves the HJB equation in the classical sense. However, it is well known that such a kind of notion turns out to be quite unsatisfactory for HJB equations arising in control theory and the calculus of variations (we refer the interested reader to the pioneer works \cite{crandalllionsevans1984someproperties,crandalllions1983viscosity} and \cite{BASCO2022126452} for further discussions). Indeed, the value function loses the differentiability property whenever there are multiple optimal solutions at the same initial condition or additional state constraints are present. The lack of classical (smooth) solutions to HJB equations for regular data led to the need of a new notion -- i.e. weak or viscosity solution -- of this equation, in the class of Lipschitz continuous functions. Such regularity is not taken for granted especially when time-dependent state constraints are imposed on the control problem \cite{bascofrankcann2017necessary,bascofrankowska2018lipschitz}.


In this paper, we focus on analyzing the Lipschitz regularity of the value function of infinite horizon control problems, specifically those with discount factors and time-dependent state constraints of a functional type. Our approach is designed to provide a comprehensive and rigorous analysis of this problem. We carefully consider the impact of the presence of time-dependent state constraints and discount factors, as these factors can significantly alter the optimal control strategy and lead to unexpected system behavior. To ensure feasibility and obtain neighboring estimates on the set of feasible trajectories, sufficient conditions on the constraint set by means of inward pointing conditions are imposed (cfr. Sections 3-4 below).
%
%
More specifically, by employing recent viability results that were investigated in \cite{BASCO2022126452}, we establish Lipschitz regularity of the value function and viability of the system. We also demonstrate that the value function vanishes at infinity on the feasible set for all sufficiently large discount factors. This result is significant, as it implies that the value function is bounded on such set, which has important implications for the stability of the system. Overall, our analysis sheds light on the behavior and regularity of weak - or viscosity - solutions of HJB equations.

The outline of the present paper is as follows. In Section 2 we describe the general formulation of the optimal control problem addressed here, with notations and backgrounds on non-smooth analysis. The Section 3 is devoted to a controllability condition on constraint set. We give a viability and neighboring estimate results in Section 4 for feasible trajectories on infinite horizon. Meanwhile in Section 5 we show the desiderate Lipschitz continuity for the value function.

\section{Problem's Formulation and Backgrounds}\label{Sec_2}
In this paper we address the following infinite horizon control problem subject to functional constraints
\begin{equation*} 
\ttnn{minimize}\quad \int_t^{+\infty} e^{-\lambda s}{\bf L}(s,x,u) ds \tag{$\ccal P$}
\end{equation*}
\eee{
\ttnn{subject to}\quad  & x'={\bf f}(s,x,u)  \quad \ttnn{a.e. }s\\
\quad & x(t)=\bar x\\
\quad &u\rs \in U\rs \quad \ttnn{a.e. }s\\
\quad  & h_1(s, x\rs )\mineq 0 \quad \forall s\mageq t \\
\quad &\vdots\\
\quad  & h_m(s, x\rs )\mineq 0 \quad \forall s\mageq t.\\
}
We assume:
\begin{itemize}
\item the controls $u$ takes values in $\bb R^m$ and are Lebesgue measurable;
\item $U$ { is a measurable set-valued map with non-empty closed images} in $\bb R^m$;
\item $h_i$'s {are real valued functions, measurable in time and space}-$\Gamma^{1,\theta}${ regular, uniformly in time};
\end{itemize}
%
where $\Gamma^{1,\theta}$ stands for the class of continuously differentiable functions with $\theta$-H\"{o}elder continuous and bounded differential, i.e., for $\theta\in]0,1[$
\eee{
\psi\in \Gamma^{1,\theta} \quad \Longleftrightarrow \quad  &\psi\in C^1 \ttnn{ with } \nabla \psi \ttnn{ bounded,}\\
&\exists k>0:\,|\nabla \psi(x)-\nabla \psi(y)| \mineq k |x-y|^\theta.
}


\noindent The optimal control problem described above is applicable to several scenarios within the fields of economics and engineering sciences (cfr. \cite{de2013optimal,feichtinger2018control,menon1990near}). In these fields of applications, functional constraints often appear as functions affine in space with measurable time-dependent terms, specifically, $h_i(s,x)=A(s)x_k+B(s)$ which falls under the framework of the proposed model. This family of functions extends to include $h_i(s,x)=A(s)\psi_i(x)+B(s)$, with $c_i\in \mathbb R^n$ a parameter and $\psi_i\in \Gamma^{1,\theta}$. It is worth to notice that the autonomous case with $\theta=1$ was previously studied in \cite{bascofrankowska2018lipschitz}.

\subsection{Preliminaries and Notations}
Let $B(x,\delta)$ stand for the closed ball in $\mathbb{R}^n$ with radius $\delta>0$ centered at $x\in \mathbb{R}^n$ and set $\mathbb{B}=B(0,1)$, $S^{n-1}=\partial \mathbb{B}$. Denote by $|\,\cdot\,|$ and $\langle \cdot , \cdot \rangle$ the Euclidean norm and scalar product, respectively. Let $C\subset \mathbb{R}^n$ be a nonempty   set. We denote the \textit{interior} of $C$ by ${\rm int}\,C$ and the \textit{convex hull} of $C$ by ${\rm co}\,C$. The \textit{distance} from $x\in \mathbb{R}^n$ to $C$ is defined by $d_C(x):=\inf\{|x-y|\,:\,y\in C\}$. If $C$ is closed, we let $\Pi_C(x)$ be the set of all \textit{projections} of $x\in \mathbb{R}^n$ onto $C$.

For $p\in \mathbb{R}^+\cup \{\infty\}$ and a Lebesgue measurable set $I\subset \mathbb{R}$ we denote by $L^p(I;\mathbb{R}^n)$ the space of $\mathbb{R}^n$-valued Lebesgue measurable functions on $I$ endowed with the norm $\|\cdot\|_{p,I}$. We say that $f\in L^p_{{\rm loc}}(I;\mathbb{R}^n)$ if $f\in L^p(J;\mathbb{R}^n)$ for any compact subset $J\subset I$.   Let $I$ be an open interval in $\mathbb{R}$. For any $f\in L^1_{{\rm loc}}(\overline I;\mathbb{R}^n)$ we define  $\theta_f:[0,\mu(I))\to \bb R^+ $ by
\eee{
\theta_{ f}(\sigma)=\sup \left\{\int_{J} |{ f}(\tau)|\,d\tau \,:\, J\subset \overline I,\, \text{Lebesgue measure of }J\leqslant  \sigma \right\}.
}
We denote by $\mathcal L_{{\rm loc}}$ the set of all functions $f\in  L^1_{{\rm loc}}(\bb R^+   ;\mathbb{R}^+)$ such that  $\lim_{\sigma\to 0}\theta_{  f}(\sigma)=0$. Notice that $L^{\infty}(\bb R^+   ;\mathbb{R}^+)\subset \mathcal L_{{\rm loc}}$  and, for any $f\in \mathcal L_{{\rm loc}}$,    $\theta_{  f}(\sigma)<\infty$ for every   $\sigma>0$.

Let   $\Omega:\bb R\rightsquigarrow\bb R^n$, $F:\overline I\times \mathbb{R}^n\rightsquigarrow \mathbb{R}^n$, and $G:\mathbb{R}^m\rightsquigarrow \mathbb{R}^n$ be set-valued maps with nonempty  values. $G$ is said to be \textit{$L$-{Lipschitz continuous}}, for some $L\geqslant 0$, if $G(x)\subset G(\tilde x)+ L|x-\tilde x|\mathbb {B}$ for all $x,\,\tilde x\in \mathbb{R}^m$.  We say that $F$ has a \textit{sub-linear growth} (in $x$) if, for some $c\in L^1_{{\rm loc}}(\overline I;\mathbb{R}^+)$, 
\eee{\sup_{v\in F(t,x)}|v|\leqslant  c (t) (1+|x|) \quad \ttnn{a.e. }t\in \overline I,\; \forall x\in \mathbb{R}^n.}
\begin{definition}
Let $\gamma\in L^1_{{\rm loc}}(\overline I;\mathbb{R}^+)$. We say that $F$ is \textit{$\gamma$-{left absolutely continuous},  uniformly wrt $\Omega$}, if
\begin{equation}\label{abs_cont_G}
F(s,x)\subset F(t,x)+\int_s^t \gamma(\tau)\,d\tau \mathbb{B}\qquad \forall s,t\in \overline I \,:\, s<t,\, \forall x\in \cup_{\tau\in[s,t]}\Omega(\tau).
\end{equation}
If $F$ does not depends explicitely from $x$, in that case we simply say that $F$ is $\gamma$-left absolutely continuous. 
\end{definition}
If $\overline I=[S,T]$, then we have the following characterization of uniform absolute continuity from the left: $F$ is $\gamma$-left absolutely continuous, uniformly wrt $\Omega$, for some $\gamma\in L^1_{{\rm loc}}(\overline I;\mathbb{R}^+)$, if and only if for every $\varepsilon >0$ there exists $\delta>0$ such that for any finite partition $S\leqslant  t_1<\tau_1\leqslant  t_2<\tau_2\leqslant  ...\leqslant  t_m<\tau_m\leqslant  T$ of $[S,T]$,
\eee{
\sum_{i=1}^m
(\tau_i-t_i)<\delta \; \Longrightarrow \; \sum_{i=1}^m exc(F(t_i,x) |F(\tau_i,x)) <\varepsilon\quad \forall \; x\in \cup_{\tau\in[S,T]}\Omega(\tau)
}
where the \textit{excess} of $A$ given $B$ is defined by
\eee{exc(A|B):=\sup\{d_B(a)|a\in A\}\in \bb R^+\cup \{+\infty\}.}

\section{Controllability}

In what follows, we take the notation
\eee{
\Omega\rt:=\bigcap_{i=1}^m \Omega_i\rt,\quad
\Omega_i(t):=\{x\in \mathbb{R}^n\,:\, h_i(t,x)\leqslant  0\}.
}

Consider the following condition
\begin{assumptions}\label{ass_h_i}
Let $\theta\in]0,1[$ and $h_i:\bb R^+ \times \mathbb{R}^n \to  \mathbb{R}$ be $m$ real-valued functions satisfying for any $i=1,...,m$:
\begin{itemize}
\item $h_i(.,x)$ is measurable for any $x$.
\item $h_i(t,.)$ is $\Gamma^{1,\theta}$ regular, uniformly wrt $t$.
\end{itemize}
\end{assumptions}

The Proposition below states a geometric result for a Inward Pointing Field Condition (also known as Inward Pointing Condition) on infinite horizon wrt the constraints $\Omega(t)$ and a vector fields $F(t,x)$.

\begin{prop}[Inward Pointing Fields Condition]\label{prop_ipc_g_i} Consider the Assumptions \ref{ass_h_i}. Let   $F:\bb R\times \bb R^n \rightsquigarrow \bb R^n$ be a set-valued map with nonempty  closed values satisfying
\eee{
&\exists M\geqslant 0, \varphi>0:\\
&(a) \; \sup\{|v|\,:\,v\in F(t,x),\, t \in \bb R^+, x\in \partial \Omega\rt\}\leqslant  M\\
&(b)\;F(t,\cdot) \ttnn{ is }\varphi \ttnn{-Lipschitz continuous for any } t\geqslant 0.
}
Assume that 
\equazioneref{viability_cond}{
&\ttnn{for some $\delta>0,\,r>0$ and for all $t\in \bb R^+ ,\,x\in \partial \Omega\rt$}\\
&\ttnn{there exists $v\in {\rm co}\,F(t,x)$ satisfying}\\
&\langle \nabla h_i(t,x), v \rangle \leqslant  -r\qquad\forall\, i\in \bigcup_{z\in B(x,\delta)} I(z)
}
where $I(z)=\{i\in I\,:\, z\in \partial \Omega_i(t)\}$ and $I:=\{1,...,m\}$. Then
\equazioneref{condition_InPoCon}{
&\ttnn{for some $\varepsilon >0,\,\eta>0$ and every $t\in \bb R^+,\,x\in (\partial \Omega(t)+\eta \mathbb{B})\cap \Omega(t)$}\\
&\ttnn{there exists $v\in {\rm co}\,F(t,x)$ satisfying}\\
&\{y+[0,\varepsilon ](v+\varepsilon  \mathbb{B})\,:\, y\in (x+\varepsilon  \mathbb{B})\cap \Omega(t)\}\subset \Omega(t).
}

\end{prop}
\begin{proof}
Let us set $J(x):=\bigcup_{z\in B(x,\delta)} I(z)$ for all $x\in \partial \Omega\rs$ and $s\mageq 0$. Fix $t\in \bb R^+,\,x\in \partial \Omega\rt$, and  $v\in {\rm co}\,F(t,x)$ satisfying $\langle \nabla h_i(t,x), v \rangle\leqslant  -r$ for all $i\in J(x)$. Pick 
\eee{
&k> \max_{i\in I} \sup_{x\neq y}\frac{|\nabla h_i(t,x)-\nabla h_i(t,y)|}{|x-y|^\theta}\\
&L> \max_{i\in I}\sup_{x\in \mathbb{R}^n}|\nabla h_i(t,x)|.
}
We proceed by steps.

{(i)}: We claim that there exists $\eta'>0$, not depending on $(t,x)$, such that for all $y\in B(x,\eta')$ we can find $w\in {\rm co}\,F(t,y)$, with $|w-v|\leqslant  r/4L$, satisfying for all $i\in J(x)$,
\eee{
\langle \nabla h_i(t,y), w \rangle \leqslant  -r/2.}
Indeed, for all $i\in J(x)$ and $y \in B(x,\sqrt[{\theta}]{r/4kM})$ we have 
\eee{
\langle\nabla h_i(t,y),v \rangle&= \langle \nabla h_i(t,y)-\nabla h_i(t,x), v \rangle+\langle\nabla h_i(t,x),v \rangle\\
&\leqslant  kM|y-x|-r\leqslant  - \frac{3r}{4} }
and
for all $w\in \mathbb{R}^n$ such that $|w-v|\leqslant  r/4L$  
\eee{
\langle \nabla h_i(t,y), w \rangle  &= \langle\nabla h_i(t,y), w-v\rangle + \langle\nabla h_i(t,y), v \rangle\\
&\leqslant  L|w-v|-3r/4\leqslant  - \frac{r}{2}.
}
  Since $F(t,\cdot)$ is $\varphi$-Lipschitz continuous,  there exists $w\in {\rm co}\,F(t,y)$ such that $|w-v|\leqslant  r/4L$ whenever $|y-x|\leqslant  r/4\varphi L$. So the claim follows with $\eta'=\min \{r/4\varphi L, \sqrt[\theta]{r/4kM}\}$.

{(ii)}: We claim that there exists $\varepsilon '>0$, not depending on $(t,x)$, such that for all $y \in B(x,\eta')$ we can find $w\in {\rm co}\,F(t,y)$ such that
\eee{
\langle\nabla h_i(t,z),\tilde w \rangle \leqslant  -r/4 \qquad \forall\, z\in B(y,\varepsilon ')\,\forall\, \tilde w \in B(w,\varepsilon ')\, \forall\, i\in J(x).
}
Indeed, let $y\in B(x,\eta')$ and $w\in{\rm co}\, F(t,y)$ be as in (i). Then for any $\tilde w\in \mathbb{R}^n$ such that $|\tilde w-w|\leqslant  r/8L$  and  for all $i\in J(x)$ and $z\in \mathbb{R}^n$,
\eee{
&\langle\nabla h_i(t,z),\tilde w \rangle \\
&=\langle\nabla h_i(t,z)-\nabla h_i(t,y),\tilde w \rangle +\langle \nabla h_i(t,y),\tilde w-w \rangle+ \langle \nabla h_i(t,y), w \rangle\\
&\leqslant  k (M+r/4L+r/8L)|z-y|+r/8-r/2.
}
So the claim follows with $\varepsilon '= \min \{k^{-1}(M+r/2L)^{-1} r/8,\,r/8L\}$.

{(iii)}: We prove that there exist $\eta>0,\,\varepsilon >0$, not depending on $(t,x)$, such that for all $y\in B(x,\eta)\cap \Omega\rt$ we can find $w\in {\rm co}\,F(t,y)$ satisfying
\begin{equation}\label{z+tau_tilde_w}
z+\tau \tilde w\in \Omega\rt\qquad \forall\, z\in B(y,\varepsilon )\cap \Omega\rt,\, \forall\, \tilde w\in B(w,\varepsilon ),\, \forall\, 0\leqslant  \tau\leqslant  \varepsilon .
\end{equation}
Let $y\in B(x,\eta')\cap \Omega\rt$ and $w\in {\rm co}\, F(t,y)$ be as in (ii). Then, by the mean value theorem, for any $\tau\geqslant 0$, any $z\in B(y,\varepsilon ')\cap  \Omega\rt$, any $\tilde w\in B(w,\varepsilon ')$, and any $i\in J(x)$ there exists $\sigma_\tau\in [0,1]$ such that
\eee{
h_i(t,z+\tau \tilde w)&=h_i(t,z)+\tau \langle \nabla h_i(t,z+\sigma_\tau \tau \tilde w), \tilde w \rangle \\
&\leqslant  \tau \langle \nabla h_i(t,z),\tilde w \rangle  +k(M+r/4L+\varepsilon ')^2\tau^2\\
&\leqslant  -\frac{r\tau}{4}+k(M+r/4L+\varepsilon ')^2\tau^2.
}
Choosing $\eta\in]0, \eta']$ and $\varepsilon \in]0, \varepsilon ']$ such that $\eta+\varepsilon (M+r/4L+\varepsilon )\leqslant  \delta$ and $\varepsilon \leqslant  {k^{-1}(M+r/4L+\varepsilon ')^{-2}r/4}$, it follows that for all $z\in B(y,\varepsilon )\cap \Omega(t),\,\tilde w\in B(w,\varepsilon )$, and all $0\leqslant  \tau\leqslant  \varepsilon $
\begin{equation}\label{uno}
z+\tau \tilde w\in B(x,\delta)
\end{equation}
and
\begin{equation}\label{UNO}
h_i(t,z+\tau \tilde w)\leqslant  0\qquad \forall\, i\in J(x).
\end{equation}
Furthermore, by (\ref{uno}) and since $B(x,\delta)\subset \Omega_j(t)$ for all $j\in I\backslash J(x)$, we have for all $z\in B(y,\varepsilon )\cap \Omega\rt,\,\tilde w\in B(w,\varepsilon )$, and all $0\leqslant  \tau\leqslant  \varepsilon $
\begin{equation}\label{DUE}
h_i(t,z+\tau \tilde w)\leqslant  0\qquad \forall\, i\in I\backslash J(x).
\end{equation}
The conclusion follows from (\ref{UNO}) and (\ref{DUE}).
\\
\end{proof}

\section{Viability and Distance Estimates on Trajectories}
We provide here sufficient conditions for uniform linear $L^\infty$ estimates on intervals of the form $I=[t_0,t_1]$, with $0\leqslant  t_0<t_1$, for the state constrained differential inclusion
\eee{
&x' (t) \in F(t,x (t) )\quad {\ttnn{a.e. } }  t\in I\\
&x (t)  \in \Omega\rt\quad \forall\, t\in I
 }
where $F:\bb R^+   \times \mathbb{R}^n \rightsquigarrow \mathbb{R}^n$ is a given set-valued map. A  function $x:[t_0,t_1] \to  \mathbb{R}^n$ is said to be:
\begin{itemize}
\item  \textit{F-trajectory} if it is absolutely continuous and $x' (t)  \in F(t,x (t) )$ for a.e. $t\in [t_0,t_1]$.
 \item \textit{feasible F-trajectory} if $x (\cdot) $ is an $F$-trajectory and $x (t)  \in \Omega\rt $ for all $  t\in I$.
\end{itemize}

\begin{assumptions}\label{Ass_F} We assume the following on $F(\cdot,\cdot)$:
\begin{enumerate} 
\item $F$ has closed and nonempty  values, a sub-linear growth, and $F(\cdot,x)$ is Lebesgue measurable for all $x\in \mathbb{R}^n$.
\item\label{M} There exist $M\geqslant 0$ and $\alpha>0$ such that
\eee{
\sup \{|v|\,:\,v \in F(t,x),\,t\in \bb R^+,\,x\in {\partial \Omega\rt +\alpha \mathbb{B}}\} \leq  M.
}
\item\label{lipF} There exists $\varphi\in \mathcal L_{{\rm loc}}$ such that $F(t,\cdot)$ is $\varphi (t) $-Lipschitz continuous for all $t\in \mathbb{R}^+$.
\item\label{AC} There exist $\tilde \eta>0$ and $\gamma\in \mathcal L_{{\rm loc}}$ such that $F$ is $\gamma$-left absolutely continuous, uniformly wrt ${\partial \Omega +\tilde \eta \mathbb{B}}$.

\end{enumerate}

\end{assumptions}

%
%
%

Before to state the main result of this section, we recall a definition and a viability result for tubes (\cite{BASCO2022126452}-Corollary 4.5).

\begin{definition}\label{def_lbv_svm}\rm
Consider a closed interval $I\subset \bb R$. We say that a set-valued map $\Phi:I\rightsquigarrow \bb R^k$ is of \textit{locally bounded variations} if:
\begin{itemize}
\item $\Phi$ takes nonempty  closed images.
\item For any $[a,b]\subset I$
\eee{\sup \;\sum_{i=1}^{m-1} exc(\Phi(t_{i+1})\cap \scr K| {\Phi(t_i)})\vee exc(\Phi(t_i)\cap \scr K|{\Phi(t_{i+1})})<+\infty}
\end{itemize}
where the supremum is taken over all compact subset $\scr K\subset \bb R^k$ and all finite partition $a=t_1< t_2< ...< t_{m-1}<t_m= b$. 
\end{definition}

In the next result, we need to recall the definition of Boulingad (or contingent) cone. Consider a closed set $G\subset\bb R^n$. The Boulingad tangent cone at $x\in G$ is defined by $\scr T_G(x):=\{ v\in \bb R^n:\exists t_i\ra 0+, \exists v_i\ra v, x+t_iv_i \in G \,\forall i \}$.

\begin{prop}[Existence of Viable Trajectories, \cite{BASCO2022126452}]\label{prop_viability}
Let $E:\bb R^+ \rightsquigarrow  \bb{R}^{d}$ be continuous\footnote{in the sense of set-valued maps, see e.g. \cite{aubin2009set}-Section 1.4.}, of locally bounded variations in the sense of Definition \ref{def_lbv_svm},  and consider $\Phi:\bb R^+ \times \bb R^d\rightsquigarrow \bb R^d$ a set-valued map with nonempty  convex closed values such that:
\eee{
-&\quad \Phi(.,x) \ttnn{ is measurable for any $x $}\\
-&\quad \exists \varrho\in  L^1_{loc}(\bb R^+ ;\bb R^+ )\\
&\quad \forall r>0\,\exists \,k_r:\halfline \ra \halfline \ttnn{ locally bounded: for a.e. } t\\
& \sup_{v\in \Phi (t,x)}|v|\mineq \varrho(t)(1+|{x}|) \ttnn{ and }
 \Phi(t,.) \textnormal{ is $k_r(t)$-Lipschitz on } r\bb B.
}
If for a.e. $t>0$ and all $y\in E(t)$ it holds
\equazioneref{non_triviality}{
{\ttnn{cl co }} \scr T_{\text {graph } E}(t, y)\cap (\{1\} \times \Phi(t,y))\neq \emptyset,}
then for any $t_0\in \bb R^+ $ and $x_0\in E(t_0)$ there exists an absolutely continuous viable solution
\eee{
&x'\rt\in \Phi(t,x\rt)\; \text{ a.e. }t>t_0\\
&x(t_0)=x_0\\
&x(t) \in E(t)\quad \forall t>t_0.
}

\end{prop}

\begin{rema}
\
\begin{enumerate}
\item Proposition \ref{prop_viability} extends classical viability results under restricted conditions on the regularity of the tube $E$ (we refer the interested reader to the bibliography therein \cite{BASCO2022126452}). Furthermore, it is straightforward to see that Lipschitz continuity for set-valued maps imply the locally bounded variations property.
\item We notice that, whenever Proposition \ref{prop_ipc_g_i} applies, then condition \rif{viability_cond} on $\Phi(t,x)=\{{\bf{f}}(t,x,u):u\in U\rt\}$ ensure the non-triviality intersection \rif{non_triviality} for $E\rt=\Omega\rt$.
 \end{enumerate}
\end{rema}

We have the following 
\begin{theo}[Neighboring Trajectories Estimates]\label{main_theo}
Consider Assumptions \ref{Ass_F}. Suppose that $h_i$'s satisfy the viability condition \rif{viability_cond} and there exists $L\mageq 0$ such that 
\equazioneref{omega_lip_cont}{\ttnn{the set-valued map $\Omega:\bb R^+\rightsquigarrow \bb R^n $ is $L$-Lipschitz continuous.}}
 Then for every $\delta>0$ there exists a constant $\beta>0$ such that for any $[t_0,t_1]\subset \bb R^+   $ with $t_1-t_0 = \delta$, any $F$-trajectory $\hat x (\cdot) $ defined on $[t_0,t_1]$ with $\hat x(t_0)\in \Omega(t_0)$, and any $\varrho>0$ satisfying
\eee{
 \sup_{t\in [t_0,t_1]}d_{\Omega(t)}(\hat x (t) )\mineq \varrho
}
we can find an $F$-trajectory $x (\cdot) $ on $[t_0,t_1]$ such that
\eee{
&x(t_0)=\hat x(t_0)\\
&\|\hat x-x\|_{\infty,[t_0,t_1]}\leqslant  \beta \varrho\\
&x (t)  \in {\rm int}\; \Omega\rt \qquad\forall\, t\in]t_0,t_1].
}
 
\end{theo}

\begin{proof}
Fix $\delta>0$ and let $[t_0,t_1]\subset \bb R^+$ with $t_1-t_0=\delta$.

We first show the statement whenever $F=\ttnn{co }F$.
Let
\begin{equation}\label{Delta}
\eps>0,\, k>0,\,\Delta>0,\,\bar \varrho>0,{\rm and } \;\;m\in \mathbb{N}^+
\end{equation}
be such that
\begin{equation}\label{cond1} \begin{split}
   \Delta\leqslant  \varepsilon,\quad
    \bar \varrho+M\Delta <\varepsilon ,\quad k\bar \varrho <\varepsilon, \quad k>1/\varepsilon,\quad
    4\Delta M\leqslant \hat \eta
 \end{split}\end{equation}
\begin{equation}\label{cond2} \begin{split}
  e^{\theta_\varphi(\Delta)}(\theta_\gamma(\Delta)+\theta_\varphi(\Delta)M)<\varepsilon,\;
  2e^{\theta_\varphi(\Delta)}(\theta_\gamma(\Delta)+\theta_\varphi(\Delta)M)k<(k\varepsilon -1)
 \end{split}\end{equation}
and
\begin{equation}\label{cond3}
\frac{\delta}{m}\leqslant  \Delta.
\end{equation}
Notice that all the constants appearing in (\ref{Delta}) do not depend on the time interval $[t_0,t_1]$, the trajectory $\hat x (\cdot) $, and $\varrho$.

1) $\varrho\leqslant  \bar \varrho$ and $\delta\leqslant  \Delta$.

\noindent We observe that, by the last inequality in (\ref{cond1}), if 
\eee{\hat x(t_0)\in \Omega(t_0) \backslash(\partial \Omega(t_0)+\frac{\hat \eta }{2}\mathbb{B})}
then $x (\cdot) =\hat x (\cdot) $ is as desired. Indeed, without loss of generality, assume $\hat x(t_0)\in (\partial \Omega(t_0)+\frac{\hat \eta }{2}S^{n-1} )\cap \Omega(t_0)$ and suppose by contraddiction that
\eee{
\ccal R:=\{ t\in ]t_0,t_1] : \hat x(t)\in  \ttnn{cl }(\bb R^n\backslash \Omega\rt )\}\neq \emptyset.
}
Put $s:=\inf R$ and notice that $s\neq t_0$. Then, we have
\eee{
d_{\partial \Omega(t_0)}(\hat x(s))\mineq \ttnn{dist}({\partial \Omega(t_0)}, \partial \Omega\rs )+d_{\partial \Omega\rs}( \hat x(s))\mineq L  |s-t_0|
}
where $\ttnn{dist}(A,B)$ stands for the standard Euclidean distance between two sets $A$ and $B$. Since
\eee{
d_{\partial \Omega(t_0)}( \hat x(t_0))-d_{\partial \Omega(t_0)}(\hat x(s))\mineq |\hat x(s)-\hat x(t_0)|
}
it follows that
\eee{
\hat \eta/4-L  |s-t_0|\mineq M\Delta\mineq \hat \eta/4,
}
a contraddiction. Next we assume that $\hat x(t_0)\in (\partial \Omega(t_0)+\frac{\hat \eta }{2}\mathbb{B})\cap \Omega(t_0) $.\\
\noindent  From Proposition \ref{prop_ipc_g_i}, let $v\in F(t_0,\hat x(t_0))$ be as in \rif{condition_InPoCon}  and define
\eee{y:[t_0,t_1] \to  \mathbb{R}^n}
 by 
\begin{equation}\label{ydef}
y(t_0)=\hat x(t_0),\quad 
y' (t) = \left\{ \begin{array}{ll}
v &t\in[t_0,(t_0+k\varrho)\wedge t_1]\\
\hat x'(t-k\varrho)&t\in]t_0+k\varrho,t_1]\cap J
\end{array}\right.
\end{equation}
where $J=\{s\in]t_0+k\varrho,t_1]\,:\,\hat x'(s-k\varrho)\; \; {\rm exists}\}$. Hence
\begin{equation}\label{stima_y_1}
\|\hat x -y\|_{\infty,[t_0,t_1]}\leqslant  2Mk\varrho.
\end{equation}
By Filippov's theorem (cfr. \cite{aubin2009set}) there exists an $F$-trajectory $x (\cdot) $ on $[t_0,t_1]$ such that $x(t_0)=y(t_0)$ and
\begin{equation}\label{filippov}
\|y-x\|_{\infty,[t_0,t]}\leqslant  e^{\int_{t_0}^t \varphi (\tau)\,d\tau}{\int_{t_0}^t d_{F(s,y(s))}(y'(s))}\,ds
\end{equation}
for all $t\in [t_0,t_1]$. Then, using Assumptions \ref{Ass_F}-\ref{lipF}, (\ref{abs_cont_G}), and (\ref{ydef}), it follows that
\begin{equation*}
\qquad \;\;d_{F(s,y(s))}(y'(s))\leqslant \left\{ \begin{array}{ll}
\theta_\gamma(\Delta)+\varphi(s)M(s-t_0) &{\rm a.e. }\; s\in[t_0,(t_0+k\varrho)\wedge t_1]\\
\varphi(s)Mk\varrho+\int_{s-k\varrho}^{s}\gamma(\tau)\,d\tau &{\rm a.e. }\; s\in]t_0+k\varrho,t_1].
\end{array}\right.
\end{equation*}
Hence, we obtain for any $t\in[t_0,(t_0+k\varrho)\wedge t_1]$
\eee{
\int_{t_0}^t d_{F(s,y(s))}(y'(s))\,ds\leqslant  (\theta_\gamma(\Delta)+\theta_\varphi(\Delta)M)(t-t_0)
}
and, using the Fubini theorem, for any $t\in]t_0+k\varrho,t_1]$
\eee{
\int_{t_0+k\varrho}^t d_{F(s,y(s))}(y'(s))\,ds\leqslant  (\theta_\varphi(\Delta)M+\theta_\gamma(\Delta))k\varrho.
}
Thus, by (\ref{filippov}), for all $t\in[t_0,(t_0+k\varrho)\wedge t_1]$
\begin{equation}\label{stima_t_t_0}
\|y-x\|_{\infty,[t_0,t]}\leqslant  e^{\theta_\varphi(\Delta)}(\theta_\gamma(\Delta)+\theta_\varphi(\Delta)M)(t-t_0)
\end{equation}
and
\begin{equation}\label{x-y}
\|y-x\|_{\infty,[t_0,t_1]}\leqslant  2e^{\theta_\varphi(\Delta)}(\theta_\gamma(\Delta)+\theta_\varphi(\Delta)M)k\varrho.
\end{equation}
Finally, taking note of (\ref{stima_y_1}), it follows that
\eee{
\|\hat x-x\|_{\infty,[t_0,t_1]}\leqslant  \beta_1 \varrho,
}
where we put $\beta_1=2 (M+e^{\theta_\varphi(\Delta)}(\theta_\gamma(\Delta)+\theta_\varphi(\Delta)M))k$.

 We claim next that
\eee{
x (t) \in {\rm int}\,\Omega\rt  \quad  \forall t\in]t_0,t_1].
}
Indeed, if  $t\in]t_0,(t_0+k\varrho)\wedge t_1]$, then from \rif{condition_InPoCon}, the first condition in (\ref{cond1}), and (\ref{ydef}) it follows that
\eee{
y (t) +(t-t_0)\varepsilon  \mathbb{B}=\hat x(t_0)+(t-t_0)(v+\varepsilon  \mathbb{B})\subset \Omega(t)
}
and it is enough to use (\ref{stima_t_t_0}) and the first inequality in (\ref{cond2}).

\noindent On the other hand, if $t\in]t_0+k\varrho,t_1]$, then for $
\pi (t)  \in \Pi_{\Omega\rt}(\hat x(t-k\varrho))
$ we have  $|\hat x(t-k\varrho)-\pi (t) |= d_{\Omega\rt}(\hat x(t-k\varrho))\leqslant  \varrho$, and, from (\ref{ydef}), it follows that
\begin{equation}\label{B1}
y (t)  \in \pi (t)  +k\varrho v +\varrho \mathbb{B}.
\end{equation}
Now, since $|\pi (t)  -\hat x(t_0)|\leqslant  |\hat x(t-k\varrho)-\pi (t) |+ |\hat x(t-k\varrho)-\hat x(t_0)| \leqslant  \bar \varrho+M\Delta$, from Proposition \ref{prop_ipc_g_i}  and the $2^{\ttnn{nd}}$ inequality in (\ref{cond1})
\begin{equation}\label{B2}
\pi (t) +k\varrho v+k\varrho \varepsilon  \mathbb{B}=\pi (t)  +k \varrho (v+\varepsilon  \mathbb{B})\subset \Omega\rt.
\end{equation}
Finally, (\ref{B1}) and (\ref{B2}) imply that
$
y (t) +(k\varepsilon -1)\varrho \mathbb{B}\subset \Omega\rt.
$
So, the claim follows from (\ref{cond2})-$(ii)$ and (\ref{x-y}).

2) $\varrho> \bar \varrho$ and $\delta\leqslant  \Delta$.

\noindent By Proposition \ref{prop_viability}, there exists a feasible $F$-trajectory $\bar x (\cdot) $ on $[t_0,t_1]$ starting from $\hat x(t_0)$. Note that $d_{\Omega\rt}(\bar x(t))=0$ for all $t\in [t_0,t_1]$. By the Case 1, replacing $\hat x (\cdot) $ with $\bar x (\cdot) $, it follows that there exists a feasible $F$-trajectory $x (\cdot) $ on $[t_0,t_1]$ such that $x(t_0)=\hat x(t_0)$ and $x\rt\in \ttnn{int }\Omega\rt$ for all $t\in]t_0,t_1]$. Hence, by Assumption \ref{Ass_F}-\ref{M}, we have
$
\|\hat x-x\|_{\infty,[t_0,t_1]}\leqslant 2 M\Delta\leqslant  \beta_2\varrho,
$
with $\beta_2=\frac{2M\Delta}{\bar \varrho}$.

3) $\delta> \Delta$.

\noindent The above proof implies that in Cases 1 and 2,  $\beta_1, \, \beta_2$ can be taken the same if $\delta$ is replaced by any $0< \delta_1 < \delta$.
Define $\tilde \beta=\beta_1\vee \beta_2$ and let $\{[\tau_-^i,\tau_+^i]\}_{i=1}^m$ be a partition of $[t_0,t_1]$ by the intervals with the length at most $\delta/m$.
Put $x_0 (\cdot) :=\hat x (\cdot) $. From Cases 1 and 2,  replacing $[t_0,t_1]$ by $[\tau_-^1,\tau_+^1]$ and setting
\eee{
\varrho_0=\varrho
}
we conclude that there exists an $F$-trajectory $x_1 (\cdot) $ on $[\tau_-^1,\tau_+^1]=[t_0,\tau_+^1]$ such that $x_1(t_0)=\hat x(t_0)$, $x_1\rt\in \ttnn{int } \Omega\rt$ for all $t\in]t_0,\tau_+^1]$, and
\eee{
\|x_1-x_0\|_{\infty,[\tau_-^1,\tau_+^1]}\leqslant  \tilde \beta\varrho_0.
}
Using Filippov's theorem, we can extend the trajectory $x_1 (\cdot) $ on whole interval $[t_0,t_1]$ so that
\eee{
\|x_1-x_0\|_{\infty,[t_0,t_1]}\leqslant  e^{\int_{t_0}^{t_1} \varphi(\tau)\,d\tau} \tilde \beta \varrho_0\leqslant  K \tilde \beta \varrho_0
}
where $K:=e^{\theta_{\varphi}(\delta)}$.
Repeating recursively the above argument on each time interval $[\tau_-^i,\tau_+^i]$, we conclude that there exists a sequence of $F$-trajectories $\{x_i (\cdot) \}_{i=1}^m$ on $[t_0,t_1]$, such that:
\begin{itemize}
\item  $x_i(t_0)=\hat x(t_0)$ for all $i=1,...,m$;
\item $x_i\rt\in \ttnn{int }\Omega\rt$ for all $t\in]t_0,\tau_+^i]$ and  all $i=1,...,m$;
\item $x_j (\cdot) |_{[t_0,\tau_+^{j-1}]}=x_{j-1} (\cdot) $ for all $j=2,...,m$;
\end{itemize}
and
\begin{equation}\label{i-i-1}
\|x_{i}-x_{i-1}\|_{\infty,[t_0,t_1]}\leqslant  K \tilde \beta \varrho_{i-1}\qquad \forall\, i=1,...,m
\end{equation}
where 
\eee{\varrho_{i-1}=\max \{\varrho, \sup_{t\in[t_0,t_1]}d_{\Omega\rt}(x_{i-1} (t) )\}.}
Notice that
\begin{equation}\label{rho_i<}
\varrho_{i}\leqslant  \varrho_{i-1}+\|x_{i}-x_{i-1}\|_{\infty,[t_0,t_1]} \qquad\forall\, i=1,...,m.
\end{equation}
Taking note of (\ref{i-i-1}) and (\ref{rho_i<}) we get for all $i=1,...,m$
\eee{
\|x_{i}-x_{i-1}\|_{\infty,[t_0,t_1]}&\leqslant  K\tilde \beta (\varrho_{i-2}+\|x_{i-1}-x_{i-2}\|_{\infty,[t_0,t_1]})\\
&\leqslant  K\tilde \beta (1+K\tilde \beta)\varrho_{i-2}\\
& \leqslant ...\\
&\leqslant  K\tilde \beta (1+K\tilde \beta)^{i-1}\varrho_0.
}
Then, letting $x (\cdot) :=x_m (\cdot) $ and observing that $\varrho_0\leqslant  \varrho$, we obtain
\eee{
\|x-\hat x\|_{\infty,[t_0,t_1]}&\leqslant \sum_{i=1}^{m}\|x_i-x_{i-1}\|_{\infty,[t_0,t_1]}\\
&\leqslant  K\tilde \beta \varrho_0 \sum_{i=1}^{m}(1+K\tilde \beta)^{i-1}\leqslant  \beta_3\varrho,
}
where  $
\beta_3=(1+K\tilde \beta)^m-1.
$

Then all conclusions of the theorem follow with \(\beta=\tilde \beta \vee\beta_3\). Observe that $\beta$ depends only on $\varepsilon ,\hat \eta $, $M$, $\delta$, and on functions $\gamma (\cdot) $ and $\varphi (\cdot) $.

Now, assume $F\neq \ttnn{co }F$. From the first part of the proof, we have that there exist $\beta>0$ (that does not depend on the reference trajectory $\hat{x}\ccd $ on $[t_0, t_1]$) and a $\ttnn{co }F$ trajectory $ \bar x \ccd :[t_0, t_1] \rightarrow \mathbb{R}^{n}$,  strictly feasible on $ ]t_0, t_1]$, such that
\eee{
\left\| \bar x -\hat{x} \right\|_{\infty, [t_0,t_1]} \leqslant \beta \varrho .
}
Let $\left\{s_{i}\right\}_i\subset  ]t_0, t_1]$ with $s_{1}=t_1$ be a decreasing sequence such that $s_{i} \ra t_0$. Since $\bar x\ccd $ is strictly feasible on $ ]t_0, t_1]$ we can find a sequence of decreasing numbers $\{\eps_{i}\}_i \subset]0, \varrho[$ such that $\eps_{i} \ra 0$ and
\equazioneref{BB1}{
 \bar x (\sigma)+\eps_{i} \mathbb{B} \subset \Omega(\sigma) \quad \forall \sigma \in\left[s_{i}, t_1\right]\; \forall i\mageq 2.
}
Without loss of generality, we can assume that $\eps_i\mineq \frac{1}{4}\wedge \varrho$ for all $i\in \bb N^+$. Put $C:=e^{ \int_{t_0}^{t_1} \varphi(\sigma) d \sigma}$ and define $a_k:=\frac{\eps_k}{C^k}$ for all $k\in \bb N$. Notice that
\eee{
\sum_{k=i}^{\infty} C^{k} a_{k}<\frac{\eps_{i}}{ 2} \quad \forall i \geqslant 2 .
}
We recall the following known relaxation result.
\begin{lemm}[Relaxation, \cite{vinter00}]Consider a measurable set-valued map $F: [S,T]\times \bb R^n \rightsquigarrow \bb R^{n}$ with closed and nonempty values. Assume that there exist $\varphi,\psi \in L^{1}(S,T;\bb R)$  such that
\eee{
F (t,  x) &\subset F (t, y)+\varphi(t) \bb B \quad \forall  t\in [S,T] \; \forall x,y\in \bb R^n\\
F(t, x) &\subset \psi(t)\bb  B \quad \forall(t, x) \in [S,T]\times \bb R^n.
}
Take any feasible $\ttnn{co }F$-trajectory $x(\cdot)$  and any $\eps>0$. Then there exists an $F$-trajectory $y(\cdot)$ that satisfies $y(S)=x(S)$ and
\eee{
\norm{y-x}_{\infty,[S,T]}<\eps .
}
\end{lemm}
\noindent From the above Lemma, there exist a sequence of $F$-trajectories $x_{i}:\left[s_{i}, t_1\right] \rightarrow \mathbb{R}^{n}$ such that, for all   $i \geqslant 2$, we have $x_{i}\left(s_{i}\right)= \bar x \left(s_{i}\right)$ and
\equazioneref{dis_min_a_k}{
\norm{x_{i} - \bar x }_{\infty, [s_i,t_1]} \leqslant a_{i} .
}
For each integer $j \geqslant 2$, we construct an $F$-trajectory $y_{j} :\left[s_{j}, t_1\right] \rightarrow \mathbb{R}^{n}$ as follows: $y_2\ccd:=x_2\ccd$ and for all $j>2$
\begin{itemize}
\item $y_{j}\ccd $ is the restriction of $x_{j}\ccd $ on  $ ]s_{j}, s_{j-1} ]$.
\item  for all $1\mineq k\mineq j-2$, $y_{j}\ccd $ restricted to $ ]s_{j-k}, s_{j-k-1}]$ is an $F$-trajectory with initial state $y_{j}\left(s_{j-k}\right)$, obtained by applying Filippov's theorem  with reference trajectory $y_{j-1}\ccd $.
\end{itemize}
Now fix an integer $j>2$. From Filippov's theorem and since $ x_{i} (s_{i} )= \bar x  (s_{i} )$,  for any $2 \leqslant i<j$ 
\eee{
\norm{y_{j}-x_{i}}_{\infty, [s_i,s_{i-1}] } &\leqslant C\left|y_{j}\left(s_{i}\right)- \bar x \left(s_{i}\right)\right|\\
\norm{y'_{j}- x'_{i}}_{1, [s_i,s_{i-1}]} &\leqslant C\left|y_{j}\left(s_{i}\right)- \bar x \left(s_{i}\right)\right|.
}
From these relations and \rif{dis_min_a_k} it follows that for each $2 \leqslant i<j$ and any  $l\in \bb N^+$ 
\equazioneref{AA1}{
\norm{y_{j}- \bar x}_{\infty, [s_i, s_{i-1}] } &\leqslant \sum_{k=i}^{j} C^{k-i} a_{k}
}
\equazioneref{AA2}{
\norm{y'_{j+l}- y'_{j}}_{1, [s_i,s_{i-1}]} &\leqslant 2 \sum_{k=i+1}^{j+l} C^{k-i} a_{k} .
}
Notice that $y_{j} (s_{j} )= \bar x  (s_{j} )$ for any $j \geqslant 2$. Hence we can extend each $F$-trajectory $y_{j}$ as an co $F$ trajectory to whole interval $[t_0, t_1]$, by setting $y_{j}(\sigma)= \bar x (\sigma)$ for $\sigma \in\left[t_0, s_{j}\right]$. Since the trajectories $\{y_{i}\}_i$ have initial value $\hat{x}(t_0)$ and owing the sub-linear growth of $F$, taking a subsequence and keeping the same notation, we have
\eee{
\exists\; \ttnn{co } F\ttnn{-trajectory }x\ccd : \; y_i\ra x \ttnn{ uniformly on }[t_0, t_1], \ttnn{ with } x(t_0)=\hat{x}(t_0).
}
We conclude to show that $x\ccd$ satisfy all the conclusions with $\beta$ replaced by $\beta+1$. Indeed, due to \rif{AA2}, for each $k \geqslant 2$ the $F$-trajectories trajectories $\{y_{i}\}_i$, restricted to $\left[s_{k}, s_{k-1}\right]$, forms a Cauchy sequence on $W^{1,1}\left(s_{k}, s_{k-1}\right)$\footnote{Here $W^{1,1}(a,b)$ stands for the space of all absolutely continuous functions on $[a,b]$ endowed with the norm $\norm{g}=g(a)+\int_a^b g'(s)ds$.}. So, it follows that the limiting co $F$-trajectory $x\ccd $ is an $F$-trajectory and, since  $\eps_{i} \leqslant \varrho$ for all $i\mageq 2$,
\eee{
\|\hat{x}-x\|_{\infty, [t_0,t_1]} \leqslant\left\|\hat{x}- \bar x \right\|_{\infty, [t_0,t_1]}+\left\|x- \bar x \right\|_{\infty, [t_0,t_1]} \leqslant  ({\beta}+1) \varrho.
}
Moreover, notice that  $x\ccd $ is strictly feasible on $]t_0,t_1]$. Indeed, consider $\sigma \in]t_0, t_1]$. We have  $\sigma \in]s_{i}, s_{i-1}]$ for some $i \geqslant 2$. From \rif{AA1}, \rif{dis_min_a_k}, and \rif{BB1} we get
\eee{
y_{j}(\sigma) \in  \bar x (\sigma)+\frac{\eps_{i}}{2} \mathbb{B} \subset \text { int } \Omega(\sigma) \quad \text { for all } j \geqslant i .
}
Since the $\{y_{j}\}_j$ converge uniformly to $x$,
\eee{
x(\sigma) \in  \bar x (\sigma)+\frac{\eps_{i}}{2} \mathbb{B} \subset \operatorname{int} \Omega(\sigma) .
}
This concludes our proof.
\end{proof}

Now, consider the following state constrained differential inclusion:
\eee{
&x' (t) \in F(t,x (t) )\quad {\ttnn{a.e. } }t\in [t_0,+\infty [\\
&x (t)  \in \Omega\rt \quad \forall\, t\in [t_0,+\infty [,
 }
where $t_0\geqslant 0$. A function $x:[t_0,+\infty [ \to  \mathbb{B}^n$ is said to be:
\begin{itemize}
\item  \textit{$F_{\infty}$-trajectory} if $x|_{[t_0,t_1]} (\cdot) $ is an \textit{F-trajectory}.
\item \textit{feasible $F_{\infty}$-trajectory} if $x|_{[t_0,t_1]} (\cdot) $ is a \textit{feasible F-trajectory} for all $t_1> t_0$.
\end{itemize}

\begin{theo}\label{theo_main_theo}
Consider Assumptions \ref{Ass_F}. Suppose that conditions in \rif{condition_InPoCon} hold true and
\eee{
\limsup_{t \to  \infty} \frac{1}{t}\int_0^t \varphi(\tau)\,d\tau<\infty.
}
Then there exist $C>1$ and $K>0$ such that for any $t_0\geqslant 0$, any $x^0,x^1\in \Omega(t_0)$, and any feasible $F_{\infty}$-trajectory $x:[t_0,+\infty [ \to  \mathbb{R}^n$, with $x(t_0)=x^0$, we can find a feasible $F_{\infty}$-trajectory $\tilde x:[t_0,+\infty [ \to  \mathbb{R}^n$, with $\tilde x(t_0)=x^1$, such that
\eee{
|\tilde x (t) -x (t)|\leqslant  Ce^{{K}t}|x^1-x^0|\qquad \forall\, t\geqslant t_0.
}
\end{theo}
\begin{proof}
Let $\delta=1$ and $\beta>0$ be as in Theorem \ref{main_theo}.
Consider $K_1>0,K_2>0,$ and $\tilde k>0$ such that
\begin{equation}\label{k1,k2}
2\beta+1 <e^{K_1}\quad{\rm and}\quad\int_{0}^{t+1}\varphi(s)\,ds\leqslant  K_2t+\tilde k \qquad\forall\, t\geqslant 0.
\end{equation}
Fix $t_0\mageq 0$, $x^0,x^1\in \Omega(t_0)$, with $x^1\neq x^0$, and a feasible ${F}_{\infty}$-trajectory $x:[t_0,+\infty [ \to  \mathbb{R}^n$ with $x(t_0)=x_0$. By Filippov's theorem, there exists an $F$-trajectory $y_0:[t_0,t_0+1] \to  \mathbb{R}^n$ such that $y_0(t_0)=x^1$ and
\eee{
\|y_0-x\|_{\infty,[t_0,t_0+1]}\leqslant  e^{\int_{t_0}^{t_0+1}\varphi(s)\,ds} |x^1-x^0|.
}
Denote by $x_0:[t_0,t_0+1] \to  \mathbb{R}^n$ the feasible $F$-trajectory, with $x_0(t_0)=x^1$, satisfying the conclusions of Theorem \ref{main_theo} with $\hat x (\cdot) =y_0 (\cdot) $. Thus
\eee{
\|x_0-y_0\|_{\infty,[t_0,t_0+1]}&\leqslant  \beta (\max_{t\in[t_0,t_0+1]} d_{\Omega\rt}(y_0 (t) )+|x^1-x^0|)\\
&\leqslant  \beta (\|y_0-x\|_{\infty,[t_0,t_0+1]}+|x^1-x^0|)\\
&\leqslant  2\beta e^{\int_{t_0}^{t_0+1}\varphi(s)\,ds}|x^1-x^0|
}
and therefore
\begin{equation}\label{stimastellina}\begin{split} 
\|x_0-x\|_{\infty,[t_0,t_0+1]}&\leqslant  \|x_0-y_0\|_{\infty,[t_0,t_0+1]}+\|y_0-x\|_{\infty,[t_0,t_0+1]}\\
&\leqslant  (2\beta+1)e^{\int_{t_0}^{t_0+1}\varphi(s)\,ds}|x^1-x^0|.
 \end{split}\end{equation}
Now, applying again Filippov's theorem on $[t_0+1,t_0+2]$, there exists an $F$-trajectory $y_1:[t_0+1,t_0+2] \to  \mathbb{R}^n$, with $y_1(t_0+1)=x_0(t_0+1)$, such that, thanks to (\ref{stimastellina}),
\begin{equation}\label{step0}
\|y_1-x\|_{\infty,[t_0+1,t_0+2]}\leqslant  (2\beta+1)e^{\int_{t_0}^{t_0+2}\varphi(s)\,ds}|x^1-x^0|.
\end{equation}
Denoting by $x_1:[t_0+1,t_0+2] \to  \mathbb{R}^n$ the feasible $F$-trajectory, with $x_1(t_0+1)=x_0(t_0+1)$, satisfying the conclusions of Theorem \ref{main_theo}, for $\hat x (\cdot) =y_1 (\cdot) $,  we deduce from (\ref{step0}), that
\begin{equation}\label{step1}
\|x_1-y_1\|_{\infty,[t_0+1,t_0+2]}\leqslant  \beta (2\beta+1)e^{\int_{t_0}^{t_0+2}\varphi(s)\,ds} |x^1-x^0|.
\end{equation}
Hence, taking note of (\ref{step0}) and (\ref{step1}),
\eee{
\|x_1-x\|_{\infty,[t_0+1,t_0+2]}\leqslant  (2\beta+1)^2 e^{\int_{t_0}^{t_0+2}\varphi(s)\,ds}|x^1-x^0|.
}
Continuing this construction, we obtain a sequence of feasible $F$-trajectories $x_i:[t_0+i,t_0+i+1] \to  \mathbb{R}^n$ such that $x_j(t_0+j)=x_{j-1}(t_0+j)$ for all $j\geqslant 1$, and
\begin{equation}\label{passoi}
\|x_i-x\|_{\infty,[t_0+i,t_0+i+1]}\leqslant  (2\beta+1)^{i+1}e^{\int_{t_0}^{t_0+i+1}\varphi(s)\,ds}|x^1-x^0|\;\; \forall\, i\in \mathbb{N}.
\end{equation}
Define the feasible ${F}_{\infty}$-trajectory $\tilde x:[t_0,+\infty [ \to  \mathbb{R}^n$ by
$
\tilde x(t):=x_i(t)$ if $t\in[t_0+i,t_0+i+1]
$
and observe that $\tilde x(t_0)=x^1$.
Let $t\geqslant t_0$. Then there exists $i\in \mathbb{N}$ such that $t\in[t_0+i,t_0+i+1]$. So, from (\ref{passoi}) and (\ref{k1,k2}), it follows that
\eee{
|\tilde x(t)-x(t)|&\leqslant  (2\beta+1)^{i+1}e^{\int_{t_0}^{t_0+i+1}\varphi(s)\,ds}|x^1-x^0|\\
&\leqslant  e^{\tilde k}(2\beta +1) e^{(K_1+K_2)(t_0+i)}|x^1-x^0|\\
& \leqslant  Ce^{K t}|x^1-x^0|,
}
where $K=K_1+K_2$ and $C=e^{\tilde k}(2\beta +1)$.
\\
\end{proof}

\section{Lipschitz Continuity}

Now we give an application of the results of previuous sections to the Lipschitz regularity of the value function for a class of infinite horizon optimal control problems subject to state constraints.

Let us consider\footnote{We recall that for a function $q\in L^1_{{\rm loc}}([t_0,+\infty [;\mathbb{R})$ the integral
\eee{\int_{t_0}^\infty q (t) \,dt:=\lim_{T \to \infty}\int_{t_0}^T q (t) \,dt,}
provided this limit exists.} the problem ($\ccal P$) stated in Section \ref{Sec_2}.

\begin{assumptions}\label{ass_n_2}
 We take the following assumptions on ${\bf f}$ and ${\bf L}$:
\begin{enumerate} 

\item For all $x\in \mathbb{R}^n$ the mappings ${\bf f}(\cdot,x,\cdot),\, {\bf L}(\cdot,x,\cdot)$ are Lebesgue-Borel measurable.
\item There exists $\alpha>0$ such that ${\bf f}$ and ${\bf L}$ are bounded functions on
\eee{\{(t,x,u)\,:\,t\geqslant 0,\, x\in (\partial \Omega\rt +\alpha \mathbb{B}),\, u\in U (t) \}.}

\item For all $(t,x)\in \bb R^+    \times \mathbb{R}^{n}$ the set
\eee{
\{({\bf f}(t,x,u), {\bf L}(t,x,u))\,:\,u\in U (t) \}
}
is closed.

\item There exist $c\in L^1_{{\rm loc}}(\bb R^+   ;\mathbb{R}^+)$ and $k\in \mathcal L_{{\rm loc}}$ such that for any $t\in \mathbb{R}^+,\, x,\,y\in \mathbb{R}^n$, and $u\in U (t) $,
\eee{
|{\bf f}(t,x,u)-{\bf f}(t,y,u)|+	| {\bf L}(t,x,u)- {\bf L}(t,y,u)|\leqslant  k (t)  |x-y|,
}
\eee{
|{\bf f}(t,x,u)|+	| {\bf L}(t,x,u)|\leqslant  c (t) (1+|x|).}
\item\label{ass_2_last} there exist $\tilde{\eta}>0$ and $\gamma \in \mathcal{L}_{\text {loc }}$ such that
\eee{t\rightsquigarrow\{({\bf f}(t,x,u), {\bf L}(t,x,u)): u\in U\rt\}}
 is $\gamma$-left absolutely continuous, uniformly wrt $ \partial \Omega+\tilde{\eta} \mathbb{B}$.

\item $\limsup_{t \to  \infty}\,\frac{1}{t}\int_0^{t} (c(s)+k(s))\,ds <\infty$.

\end{enumerate}
\end{assumptions}

We consider, for any $\lambda>0$, the relaxed infinite horizon state constrained problem
\begin{equation}
\ttnn{minimize } \int_{t}^{\infty}     e^{-\lambda s}{\bf L}^\star(s,x   ,w   )\,ds,\tag{$\ccal P^\star$}
\end{equation}
\eee{
\ttnn{subject to}\quad &x'   =  {\bf f}^\star(s,x ,w ) \quad {\ttnn{ a.e. }} s \\
\quad & x(t)=\bar x\\
\quad & w (s) \in W (s) \quad {\ttnn{ a.e. }} s \\
\quad  & h_1(s, x\rs )\mineq 0 \quad \forall s\mageq t \\
\quad &\vdots\\
\quad  & h_m(s, x\rs )\mineq 0 \quad \forall s\mageq t.\\
}
where
\begin{itemize}
\item $W:\bb R^+    \rightsquigarrow \mathbb{R}^{(n+1)m}\times \mathbb{R}^{n+1}$ is the measurable set-valued map defined by:
\eee{
W (s) :=(\times_{i=0}^n U (s) ) \times \{(\alpha_0,...,\alpha_n)\in \mathbb{R}^{n+1}\,:\, \sum_{i=0}^n \alpha_i=1,\, \alpha_i\geqslant 0\,\, \forall\, i\}
}
for all $s \geqslant 0$.
\item $  {\bf f}^\star: \bb R^+    \times \mathbb{R}^n\times \mathbb{R}^{(n+1)m}\times \mathbb{R}^{n+1} \to  \mathbb{R}^n$ and $   {\bf L}^\star: \bb R^+    \times \mathbb{R}^n\times \mathbb{R}^{(n+1)m}\times \mathbb{R}^{n+1} \to  \mathbb{R}$ are defined by: \eee{
  {\bf f}^\star(s,x,w)&:=\sum_{i=0}^n \alpha_i {\bf f}(s,x,u_i)\\
 {\bf L}^\star(s,x,w)&:=\sum_{i=0}^n\alpha_i  {\bf L}(s,x,u_i)
}
for all $s \geqslant 0$, $x\in \mathbb{R}^n$, and $w=(u_0,...,u_n,\alpha_0,...,\alpha_n)\in\mathbb{R}^{(n+1)m}\times \mathbb{R}^{n+1}$.
\end{itemize}

\begin{rema}
\

\begin{enumerate}
\item For control systems, the condition \rif{condition_InPoCon}   take the following form: for some $\varepsilon >0,\,\eta>0$ and every $t\in \bb R^+,\,x\in (\partial \Omega\rt+\eta \mathbb{B})\cap \Omega\rt$ there exist $\{\alpha_i\}_{i=0}^n\subset [0,1]$, with $\sum_{i=0}^n\alpha_i=1$, and $\{u_i\}_{i=0}^n\subset U(t)$ satisfying
\eee{
	 \{y+[0,\varepsilon ] (\sum_{i=0}^n\alpha_i{\bf f}(t,x,u_i)+\varepsilon  \mathbb{B} )\,:\, y\in (x+\varepsilon  \mathbb{B})\cap \Omega\rt  \}\subset \Omega\rt.
}

\item If there exist $\tilde \eta>0,\,\gamma,\,\tilde \gamma\in \mathcal{L}_{{\rm loc}}$, and $k\geqslant 0$ such that $({\bf f} ,{\bf L} )$ is $\gamma$-left absolutely continuous, uniformly wrt $ (\partial \Omega+\tilde \eta \mathbb{B})\times \mathbb{R}^m$, $U (\cdot) $ is $\tilde \gamma$-left absolutely continuous, and ${\bf f}(t,x,\cdot)$ is $k$-Lipschitz continuous for all $t\in \bb R^+,\,x\in    (\partial \Omega\rt+\tilde \eta \mathbb{B})$, then Assumption \ref{ass_n_2}-\ref{ass_2_last} holds true.
 
\end{enumerate}
\end{rema}

\begin{definition}
We denote by
\eee{
V:\ccal Q_\Omega \rightarrow \bb R\cup\{\pm \infty\}\quad \ttnn{and} \quad V^\star:\ccal Q_\Omega \rightarrow \bb R\cup\{\pm \infty\}
}
the value functions of the infinite horizon control problems ($\ccal P$) and ($\ccal P^\star$), respectively, where
\eee{
\ccal Q_\Omega:=\{(t,x)\in \bb R^+\times \bb R^n: t\in \bb R^+,\, x\in \Omega\rt\}.
}
\end{definition}

Next, we state the main result of this section

\begin{theo}\label{relaxed}Consider Assumptions \ref{ass_n_2}. Suppose that \rif{viability_cond} and \rif{omega_lip_cont} hold true. 
Then there exist $b>1$ and $K>0$ such that for all $\lambda >K$ we have
\begin{enumerate}
\item[(i)]  $ V^\star(t,\cdot)$ is $ b \cdot e^{-(\lambda-K)t}$-Lipschitz continuous on $\Omega\rt$, for any  $t\geqslant 0$.

\item[(ii)]  $\lim_{t \to  \infty} V^\star(t,x (t) )=0$ for any feasible trajectory $x (\cdot) $.

\item[(iii)] $V^\star  =V$ on $\ccal Q_\Omega$.
\end{enumerate}
\end{theo}

\begin{proof} 
We notice that, by   Proposition \ref{prop_viability}, the problem ($\ccal P^\star$) admits feasible trajectory-control pairs for any initial condition; using the sub-linear growth of ${\bf f}$ and the Gronwall lemma, we have $1+|x (t) |\leqslant  (1+|x_0|)e^{\int_{t_0}^{t}c(s)\,ds}$ for all $t\geqslant t_0$ and for any trajectory-control pair $(x (\cdot) ,u (\cdot) )$ at $t_0\in \bb R^+,\,x_0\in \Omega(t_0)$.

\noindent In what follows, we define for all $ (t,x,z)\in \bb R^+   \times \mathbb{R}^n\times \mathbb{R} $ the time-measurable set-valued maps
\eee{ 
 G(t,x,z)&=\{({\bf f}(t,x,u),e^{-\lambda t} {\bf L}(t,x,u))\,:\,u\in U (t) \}\\
G^\star(t,x,z)&=\{({\bf f}^\star(t,x,u), {\bf L}^\star (t,x,w))\,:\,w\in W (t) \}.
}
Next, we show $(i)$. Let $a_1>0,\,a_2>0$ be such that
\begin{equation}\label{a_1_a_2}
\int_{0}^{t}c(s)\,ds\leqslant  a_1t+a_2\qquad \forall t\geqslant 0.
\end{equation}
For all $T>t_0$, we have
\begin{equation}\label{prima_stima_Gronwall} \begin{array}{ll}
\int_{t_0}^{T}e^{-\lambda t}| {\bf L}^\star(t,x (t) ,w (t) )|\,dt & \leqslant (n+1) \int_{t_0}^{T}e^{-\lambda t} c (t) (1+|x_0|)e^{\int_{t_0}^{t}c(s)\,ds}\,dt\\
& \leqslant  (n+1)(1+|x_0|)e^{a_2}\int_{t_0}^{T}e^{-(\lambda -a_1)t}c (t) \,dt.
\end{array}\end{equation}
Then, by (\ref{a_1_a_2}) and denoting $\psi (t) =\int_{t_0}^{t}c(s)\,ds$, for any $\lambda >a_1$
\begin{equation}\label{seconda_stima_int_c} \begin{array}{ll}
	&\int_{t_0}^{T}e^{-\lambda t}| {\bf L}^\star(t,x (t) ,w (t) )|\,dt\\
	&\leqslant (n+1)(1+|x_0|)e^{a_2}\left( \left[e^{-(\lambda -a_1)t} \psi (t) \right]_{t_0}^T+(\lambda-a_1)\int_{t_0}^{T}e^{-(\lambda -a_1)t}{\psi (t) }\,dt \right)\\
	&\leqslant (n+1)(1+|x_0|) e^{a_2}\left(e^{-(\lambda -a_1)T}(a_1T+a_2)+\left(a_1t_0+\frac{a_1}{\lambda -a_1}+a_2\right)e^{-(\lambda -a_1)t_0} \right).
	\end{array}  \end{equation}
Passing to the limit when $T \to  \infty$, we deduce that for every feasible trajectory-control pair $(x (\cdot) ,w (\cdot) )$ at $(t_0,x_0)$ 
\eee{
\int_{t_0}^{\infty}e^{-\lambda t}| {\bf L}^\star(t,x (t) ,w (t) )|\,dt<+\infty\qquad \forall \lambda>a_1.
}

\noindent From now on, assume that $\lambda >a_1$. Fix $t\geqslant 0$ and $x^1,x^0\in \Omega\rt$ with $x^1\neq x^0$. Then, for any $\delta>0$ there exists a feasible trajectory-control pair $(x_\delta (\cdot) ,w_\delta (\cdot) )$ at $(t,x^0)$ such that
\eee{
 V^\star(t,x^0)+e^{- \delta t}|x^1-x^0|>\int_{t}^{\infty}e^{-\lambda s} {\bf L}^\star(s,x_\delta(s),w_\delta(s))\,ds.
}
Hence
\begin{equation}\label{prima_stima}\begin{array}{ll}
& V^\star(t,x^1)- V^\star(t,x^0) \leqslant e^{- \delta t}|x^1-x^0| +\\
&  \lim_{\tau \to  \infty}\left|\int_{t}^{\tau}e^{-\lambda s} {\bf L}^\star(s, x(s), w(s))\,ds-\int_{t}^{\tau}e^{-\lambda s} {\bf L}^\star(s,x_\delta(s),w_\delta(s))\,ds\right|
\end{array} \end{equation}
for any feasible trajectory-control pair $( x (\cdot) , w (\cdot) )$ satisfying $x(t)=x^1$. Consider the following state constrained differential inclusion in $\mathbb{R}^{n+1}$
\eee{
&(x,z)' (s) \in  G^\star(s,x (s) ,z (s) ) \quad  {\ttnn{ a.e. }}s\in[t,+\infty [\\
&x (s) \in \Omega\rs \quad \forall \, s\in[t,+\infty [.
}
Putting $z_\delta(s)=\int_{t}^{s} {\bf L}^\star(\xi,x_\delta(\xi),w_\delta(\xi))\,d\xi$, by Theorem \ref{theo_main_theo} applied on $\Omega\rt\times \mathbb{R}$ and the measurable selection theorem, there exist $C>1$ and $K>0$ such that for all $\delta>0$ we can find a $ G^\star_{\infty}$-trajectory $(\tilde x_\delta (\cdot) , \tilde z_\delta (\cdot) )$  on $[t,+\infty [$, and a measurable selection $\tilde w_\delta(s)\in W (s) $ a.e. $s\geqslant t$, satisfying
\eee{
&(\tilde x_\delta,\tilde z_\delta)' (s) =({\bf f}^\star(s,\tilde x_\delta (s) , \tilde w_\delta (s) ), {\bf L}^\star(s,\tilde x_\delta(s),\tilde w_\delta(s)))\qquad {\rm a.e. }\; s\geqslant t,\\
&(\tilde x_\delta(t),\tilde z_\delta(t))=(x^1,0)\\
&\tilde x_\delta(s)\in \Omega\rs \quad \forall s\in [t,+\infty [
}
and for any $s\geqslant t$
\begin{equation}\label{stimachiave}
|\tilde x_\delta (s) -x_\delta (s) |+|\tilde z_\delta (s) -z_\delta (s) |\leqslant  Ce^{K s}|x^1-x^0|.
\end{equation}
Now, relabelling by $K$ the constant $K\vee a_1$, by (\ref{stimachiave}) and integrating by parts, for all $\lambda >K$, all $\tau\geqslant t$, and all $\delta>0$
\begin{equation}\label{stimapreliminare} \begin{split}
&\left|\int_{t}^{\tau}e^{-\lambda s} {\bf L}^\star(s, \tilde x_\delta(s),\tilde w_\delta(s))\,ds-\int_{t}^{\tau}e^{-\lambda s} {\bf L}^\star(s,x_\delta(s),w_\delta(s))\,ds \right|\\
&\leqslant  \left|\left[e^{-\lambda s}\left(\int_t^s  {\bf L}^\star(\xi,\tilde x_\delta(\xi),\tilde w_\delta(\xi))\,d\xi-\int_t^s  {\bf L}^\star(\xi, x_\delta(\xi),w_\delta(\xi))\,d\xi \right)\right]_t^\tau \right|\\[2mm]
&\quad+\lambda \left| \int_t^\tau e^{-\lambda s}\left(\int_t^s  {\bf L}^\star(\xi,\tilde x_\delta(\xi),\tilde w_\delta(\xi))\,d\xi-\int_t^s  {\bf L}^\star(\xi, x_\delta(\xi),w_\delta(\xi))\,d\xi \right)\,ds \right|\\[2mm]
&\leqslant  e^{-\lambda \tau}|\tilde z_\delta(\tau)-z_\delta(\tau)|+\lambda \int_t^\tau e^{-\lambda s}|\tilde z_\delta (s) -z_\delta (s) | ds \\[2mm]
&\leqslant  Ce^{-\lambda \tau}e^{K\tau}|x^1-x^0|+\lambda C \int_{t}^\tau e^{-(\lambda -K)s}|x^1-x^0|\,ds\\[2mm]
&=\left(Ce^{-(\lambda-K)\tau}+\lambda C \left[-\frac{e^{-(\lambda-K)s}}{\lambda -K}\right]_t^\tau \right)|x^1-x^0| \\[2mm]
&=\left(-\frac{CK}{\lambda -K}e^{-(\lambda -K)\tau}+\frac{\lambda C}{\lambda -K}e^{-(\lambda-K)t}\right)|x^1-x^0|\\
&\leqslant {\frac{\lambda C}{\lambda -K}e^{-(\lambda-K)t}}|x^1-x^0|.
\end{split} \end{equation}
Taking note of (\ref{prima_stima}), (\ref{stimapreliminare}), and putting $\delta=\lambda-K$, for all $\lambda >K$ we get
\eee{
 V^\star(t,x^1)- V^\star(t,x^0)
\leqslant  \left(\frac{\lambda C}{\lambda -K}+1\right)e^{-(\lambda-K)t}|x^1-x^0|.
}
By the symmetry of the previous inequality with respect to $x^1$ and $x^0$, and since $\lambda$, $C$, and $K$ do not depend on $t$, $x^1$, and $x^0$, the statement $(i)$ follows.

Now, let $(t_0,x_0)\in \ccal Q_\Omega$ and consider a feasible trajectory $X (\cdot) $ at $(t_0,x_0)$. Let $t> t_0$ and $(x (\cdot) ,w (\cdot) )$ be a feasible trajectory-control pair at $(t,X (t) )$ such that $ V^\star(t,X (t) )>\int_{t}^{\infty}e^{-\lambda s} {\bf L}^\star(s,x (s) ,w (s) )\,ds-\frac{1}{t}$. Then 
\eee{| V^\star(t,X (t) )|\leqslant  \int_{t}^{\infty}e^{-\lambda s}| {\bf L}^\star(s,x (s) ,w (s) )|\,ds+\frac{1}{t}.
}
From (\ref{a_1_a_2}) and (\ref{prima_stima_Gronwall}), we have for all $T>t$
\eee{
&\int_{t}^{T}e^{-\lambda s}| {\bf L}^\star(s,x (s) ,w (s) )|\,ds \leqslant  \int_{t}^{T}e^{-\lambda s} (1+|X(t)|)e^{\int_{t}^{s}c(\omega)\,d\omega}c (s) \,ds.\\
&\leqslant  (1+|x_0|)\int_{t}^{T}e^{-\lambda s} e^{\int_{t_0}^{t}c(\omega)\,d\omega}e^{\int_{t}^{s}c(\omega)\,d\omega}c (s) \,ds\\
&\leqslant (1+|x_0|)\int_{t}^{T}e^{-\lambda s} e^{\int_{0}^{s}c(\omega)\,d\omega}c (s) \,ds\\
& \leqslant  (1+|x_0|)e^{a_2}\int_{t}^{T}e^{-(\lambda-a_1) s} c (s) \,ds.
}
Then, arguing as in (\ref{seconda_stima_int_c}) with $t_0$ replaced by $t$ and taking the limit when $T \to  \infty$, we deduce that
\eee{
	| V^\star(t,X (t) )|\leqslant  (1+|x_0|)e^{a_2} \left(a_1t+\frac{a_1}{\lambda -a_1}+a_2\right)e^{-(\lambda -a_1)t}+\frac{1}{t}.
}
Since $K\geqslant a_1$, $(ii)$  follows passing to the limit when $t \to + \infty$.

Next, we show $(iii)$. Notice that $ V^\star (t,x)\leqslant  V(t,x)$ for any $(t,x)\in \ccal Q_\Omega$, and   $ V^\star (t,\cdot)$   is Lipschitz continuous on $\Omega\rt$ for all $t\geqslant 0$ whenever $\lambda>0$ is sufficiently large. Fix $t_0\in \bb R^+,\,x_0\in  \Omega(t_0)$, and $\eps>0$. We claim that: for all $j\in \mathbb{N}^+$ there exists a finite set of trajectory-control pairs
$\{(x_k (\cdot) ,u_k (\cdot) )\}_{k=1,...,j}$
satisfying the following: $x_k' (s) ={\bf f}(s,x_k' (s) ,u_k' (s) )$ a.e. $s\in [t_0,t_0+k]$ and $x_k(s)\in \Omega\rs$ for all $s\in [t_0,t_0+k]$ and for all $k=1,...,j$; if $j\geqslant 2$, $x_k|_{[t_0,t_0+k-1]} (\cdot) =x_{k-1} (\cdot) $ for all $k=2,...,j$; and for all $k=1,...,j$
\begin{equation}\label{tesi_k}
	 V^\star (t_0,x_0)\geqslant  V^\star (t_0+k,x_k(t_0+k))+\int_{t_0}^{t_0+k} e^{-\lambda t} {\bf L}(t,x_k(t),u_k(t))\,dt-\varepsilon \sum_{i=1}^{k}\frac{1}{2^i}.
\end{equation}
We prove the claim by the induction argument with respect to $j\in \mathbb{N}^+$. By the dynamic programming principle, there exists a trajectory-control pair $(\tilde x (\cdot) ,\tilde w (\cdot) )$ on $[t_0,t_0+1]$, feasible for the problem ($\ccal P^\star$) at $(t_0,x_0)$, such that
\begin{equation}\label{passo0}
	 V^\star (t_0,x_0)+\frac{\varepsilon }{4}> V^\star (t_0+1,\tilde x(t_0+1))+\int_{t_0}^{t_0+1}e^{-\lambda t}   {\bf L}^\star(t,\tilde x (t) ,\tilde w (t) )\,dt.
\end{equation}
By the relaxation theorem for finite horizon problems (cfr. \cite{vinter00}), for any $h>0$ there exists a measurable control $\hat u^h(t)\in U(t)$ a.e. $t\in [t_0,t_0+1]$ such that the solution of the equation $(\hat x^h)' (t) ={\bf f}(t,\hat x^h (t) ,\hat u^h (t) )$ a.e. $t\in [t_0,t_0+1]$, with $\hat x^h(t_0)=x_0$, satisfies
\eee{
	\|\hat x^h-\tilde x\|_{\infty,[t_0,t_0+1]}< h
}
and
\eee{
	\left|\int_{t_0}^{t_0+1}e^{-\lambda t}   {\bf L}^\star(t,\tilde x (t) ,\tilde w (t) )\,dt-\int_{t_0}^{t_0+1}e^{-\lambda t} {\bf L}(t,\hat x^h (t) ,\hat u^h (t) )\,dt\right| < h.
}
Now, consider the following state constrained differential inclusion in $\mathbb{R}^{n+1}$
\eee{
	(x,z)' (s) &\in   G(s,x (s) ,z (s) ) \quad {\rm a.e. }\; s\in[t_0,t_0+1]\\
	x (s) &\in \Omega\rs \quad \forall\, s\in[t_0,t_0+1].
}
Letting $\hat X^h(\cdot)=(\hat x^h(\cdot),\hat z^h(\cdot))$, with $\hat z^h (t) =\int_{t_0}^{t} e^{-\lambda s} {\bf L}(s,\hat x^h (s) ,\hat u^h (s) )\,ds$, by Theorem \ref{main_theo}   and the measurable selection theorem, there exist $\beta >0$ (not depending on $(t_0,x_0)$) such that for any $h>0$ we can find a feasible $  G$-trajectory $X^h (\cdot) =(x^h (\cdot) ,z^h (\cdot) )$ on $[t_0,t_0+1]$, with $X^h(t_0)=(x_0,0)$, and a measurable control $u^h (s) \in U (s) $ a.e. $s\in [t_0,t_0+1]$, such that
\eee{
	&( x^h, z^h)'(s)=({\bf f}(s, x^h (s) ,  u^h (s) ),e^{-\lambda s} {\bf L}(s, x^h(s), u^h(s)))\\
	&\qquad {\rm a.e. }\; s\in[t_0,t_0+1]
	}
and
\eee{
	\|X^h-\hat X^h\|_{\infty,[t_0,t_0+1]} \leqslant  \beta ( \sup_{s\in [t_0,t_0+1]}d_{\Omega\rs\times \mathbb{R}}(\hat X^h(s))+h).
}
Since $\sup_{s\in [t_0,t_0+1]}d_{\Omega\rs\times \mathbb{R}}(\hat X^h(s))\leqslant  \|\tilde x-\hat x^h\|_{\infty,[t_0,t_0+1]}$, we have
\eee{
&\left|\int_{t_0}^{t_0+1} e^{-\lambda t} {\bf L}(t, x^h (t) , u^h (t) )\,dt-\int_{t_0}^{t_0+1}e^{-\lambda t}   {\bf L}^\star(t,\tilde x (t) ,\tilde w (t) )\,dt\right|\\
&\leqslant  \left|\int_{t_0}^{t_0+1}e^{-\lambda t}   {\bf L}^\star(t,\tilde x (t) ,\tilde w (t) )\,dt-\int_{t_0}^{t_0+1}e^{-\lambda t} {\bf L}(t,\hat x^h (t) ,\hat u^h (t) )\,dt\right|\\
&\qquad+\left|\int_{t_0}^{t_0+1} e^{-\lambda t} {\bf L}(t, x^h (t) , u^h (t) )\,dt-\int_{t_0}^{t_0+1}e^{-\lambda t} {\bf L}(t,\hat x^h (t) ,\hat u^h (t) )\,dt \right|\\
&< h(2\beta+1)
}
and
\eee{
	\| x^h-\tilde x\|_{\infty,[t_0,t_0+1]}&\leqslant  \|\tilde x-\hat x^h\|_{\infty,[t_0,t_0+1]}+\| x^h-\hat x^h\|_{\infty,[t_0,t_0+1]}\\
	&< h(2\beta+1).
}
Hence, choosing $0<h<\varepsilon /4(2\beta+1)$ sufficiently small, we can find a trajectory-control pair $(x^h (\cdot) ,u^h (\cdot) )$ on $[t_0,t_0+1]$, with $u^h (s) \in U (s) $ and $(x^h)' (s) ={\bf f}(s,x^h (s) ,u^h (s) )$ a.e. $s\in [t_0,t_0+1]$, $x^h(t_0)=x_0$, and $x^h(s)\in \Omega\rs$ for $s\in [t_0,t_0+1]$, such that, by (\ref{passo0}) and continuity of $ V^\star (t_0+1,\cdot)$
\eee{
 V^\star (t_0,x_0)> V^\star (t_0+1, x^h(t_0+1))+\int_{t_0}^{t_0+1} e^{-\lambda t} {\bf L}(t, x^h (t) , u^h (t) )\,dt-\frac{\varepsilon }{2}.
}
Letting $(x_1 (\cdot) ,u_1 (\cdot) ):=(x^h (\cdot) ,u^h (\cdot) )$, the conclusion follows for $j=1$. Now, suppose we have shown that there exist $\{(x_k (\cdot) ,u_k (\cdot) )\}_{k=1,...,j}$ satisfying the claim. Let us to prove it for $j+1$. By the dynamic programming principle there exists a trajectory-control pair $(\tilde x (\cdot) ,\tilde w (\cdot) )$ on $[t_0+j,t_0+j+1]$, feasible for the problem ($\ccal P^\star$) at $(t_0+j,x_j(t_0+j))$, such that
\begin{equation}\label{passo_j+1}  \begin{array}{ll}
	 V^\star (t_0+j,x_j(t_0+j))+\frac{\varepsilon }{2^{j+2}}&> V^\star (t_0+j+1,\tilde x(t_0+j+1))\\
	&\qquad+\int_{t_0+j}^{t_0+j+1}e^{-\lambda t}   {\bf L}^\star(t,\tilde x (t) ,\tilde w (t) )\,dt.
\end{array} \end{equation}
As before, for every $h>0$ there exist a feasible $  G$-trajectory $X^h (\cdot) =(x^h (\cdot) ,z^h (\cdot) )$ on $[t_0+j,t_0+j+1]$, with $X^h(t_0)=(x_j(t_0+j),0)$, and a measurable control $u^h (s) \in U (s) $ a.e. $s\in [t_0+j,t_0+j+1]$, such that 
\eee{
&( x^h, z^h)'(s)=({\bf f}(s, x^h (s) ,  u^h (s) ),e^{-\lambda s} {\bf L}(s, x^h(s), u^h(s))) \\
&\qquad {\rm a.e. }\;s\in[t_0+j,t_0+j+1]
} satisfying
\eee{
	&\left|\int_{t_0+j}^{t_0+j+1} e^{-\lambda t} {\bf L}(t, x^h (t) , u^h (t) )\,dt-\int_{t_0+j}^{t_0+j+1}e^{-\lambda t}   {\bf L}^\star(t,\tilde x (t) ,\tilde w (t) )\,dt \right|\\
	& <h(2\beta+1)
}
and
\eee{
	\| x^h-\tilde x\|_{\infty,[t_0+j,t_0+j+1]}<h(2\beta+1).
}
Putting
\begin{equation}\label{x_j+1}
(x_{j+1} (\cdot) ,u_{j+1} (\cdot) ):= \left\{ \begin{array}{ll}
	(x_{j} (\cdot) ,u_j (\cdot) ) &{\rm on }\;\;[t_0,t_0+j]\\
	(x^h (\cdot) ,u^h (\cdot) )&{\rm on }\;\;[t_0+j,t_0+j+1]
\end{array} \right.
\end{equation}
and choosing $0<h<\varepsilon /2^{j+2}(2\beta+1)$ sufficiently small, it follows from (\ref{passo_j+1}) that
\begin{equation}\label{Vpassoj+1} \begin{array}{ll}
	 V^\star (t_0+j,x_j(t_0+j))&\geqslant  V^\star (t_0+j+1, x_{j+1}(t_0+j+1))\\
	&\qquad+\int_{t_0+j}^{t_0+j+1}e^{-\lambda t}  {\bf L}(t, x_{j+1} (t) , u_{j+1} (t) )\,dt-\frac{2\varepsilon }{2^{j+2}}.
\end{array}  \end{equation}
So, taking note of (\ref{x_j+1}) and (\ref{Vpassoj+1}), we obtain
\eee{
	& V^\star (t_0,x_0)\\
	&\geqslant  V^\star (t_0+j,x_j(t_0+j))+\int_{t_0}^{t_0+j}e^{-\lambda t}  {\bf L}(t,x_j(t),u_j(t))\,dt-\varepsilon \sum_{i=1}^{j}\frac{1}{2^i}\\
	&\geqslant  V^\star (t_0+j+1, x_{j+1}(t_0+j+1)) -\varepsilon \sum_{i=1}^{j}\frac{1}{2^i}-\frac{\varepsilon }{2^{j+1}}\\
	&\qquad+\int_{t_0+j}^{t_0+j+1}e^{-\lambda t} {\bf L}(t, x_{j+1} (t) , u_{j+1} (t) )\,dt\\
	&\qquad+\int_{t_0}^{t_0+j}e^{-\lambda t}  {\bf L}(t,x_j(t),u_j(t))\,dt\\[2mm]
	&= V^\star (t_0+j+1, x_{j+1}(t_0+j+1))\\
	&\qquad+\int_{t_0}^{t_0+j+1}e^{-\lambda t}  {\bf L}(t,x_{j+1}(t),u_{j+1}(t))\,dt-\varepsilon \sum_{i=1}^{j+1}\frac{1}{2^i}.
}
Hence $\{(x_k (\cdot) ,u_k (\cdot) )\}_{k=1,...,j+1}$ also satisfy our claim.
Now, let us define the trajectory-control pair $(x (\cdot) ,u (\cdot) )$ by $(x (t) ,u (t) ):=(x_k (t) ,u_k (t) )$ if $t\in [t_0+k-1,t_0+k]$. Then $(x (\cdot) ,u (\cdot) )$ is a feasible trajectory-control pair for the problem ($\ccal P$) at $(t_0,x_0)$. Since, by $(ii)$,   $ V^\star (t,x (t) ) \to  0$ when $t \to  +\infty$, from (\ref{tesi_k}) we have
\eee{
 V^\star (t_0,x_0)
&\geqslant \int_{t_0}^{\infty}e^{-\lambda t}  {\bf L}(t,x(t),u(t))\,dt-\varepsilon .
}
Hence, we deduce that $(t_0,x_0)$ lays in the domain of the value function $V$ and so 
$ V^\star (t_0,x_0)\geqslant V(t_0,x_0)-\varepsilon $. From the arbitrariness of $\varepsilon $, the conclusion follows.
\end{proof}

\begin{coro}
Consider any $N>0$  with
\eee{
	N \geq \sup\{|{\bf f}(t,x,u)|+| {\bf L}(t,x,u)|\,:\, t\geqslant 0,\,x\in \mathbb{R}^n,\, u\in U (t) \}<\infty.
}
Then, for any $\lambda>0$ sufficiently large, for any $t\geqslant 0$ and any $x\in \Omega\rt$, the function $V(\cdot,x)$ is Lipschitz continuous on $[t,+\infty [$ with constant $\left(L(t)+2e^{-\lambda t}\right)N$ and $L(t):=b e^{-(\lambda-K)t}$.
\end{coro}

\begin{proof}
From Theorem \ref{relaxed}, when $\lambda>0$ is large enough, $V(t,\cdot)$ is $L (t) $-Lipschitz continuous on $\Omega\rt$. Fix $t\geqslant 0$ and  $x\in \Omega\rt$. Let $s,\tilde s\in [t,+\infty [$.

Suppose that $s\geqslant \tilde s$. Then, by the dynamic programming principle, there exists a feasible trajectory-control pair $(\bar x (\cdot) ,\bar u (\cdot) )$ at $(\tilde s,x)$ such that
\begin{equation}\label{s_mag_tilde_s}\begin{split} 
	V(s,x)-V(\tilde s,x)&\leqslant  |V(s,x)-V(s,\bar x(s))|+\int_{\tilde s}^{s}e^{-\lambda \xi}| {\bf L}(\xi,\bar x(\xi),\bar u(\xi))|\,d\xi\\
	&\qquad+N|s-\tilde s|e^{-\lambda t}\\
	&\leqslant  L(s)N|s-\tilde s|+N|s-\tilde s|e^{-\lambda \tilde s}+N|s-\tilde s|e^{-\lambda t}\\
	&\leqslant  {\left(L(t)+2e^{-\lambda t}\right)N}|s-\tilde s|.
 \end{split} \end{equation}
Arguing in a similar way, we get (\ref{s_mag_tilde_s}) when $s<\tilde s$. Hence, by the symmetry with respect to $s$ and $\tilde s$ in (\ref{s_mag_tilde_s}), the conclusion follows.
\\
\end{proof}

\begin{rema}
%
The relaxation result in Theorem \ref{relaxed} assumes crucial significance when convex data assumptions are absent. By transitioning to the relaxed problem, we ensure that both convergence and Lipschitz regularity remain guaranteed. This approach naturally aligns with the need of machine learning algorithms, in which the desiderable property of Lipschitz regularity of the value function improve significantly convergence rates (cfr. \cite{bertsekas2012dynamic,bertsekas2019reinforcement}). As mentioned earlier in the Introduction, it is well known that incorporating state constraints in the learning process can introduce instability or lead to error divergence in function approximation techniques.
Hence, to bolster the overall robustness and reliability of the methods, inward point conditions such as those in \rif{viability_cond} play a critical role. We refere the reader to \cite{de2016optimal,munos1998general, munos2000study, weston2022mixed} for a more comprehensive understanding of the roles of inward pointing conditions and Lispschitz continuity of value functions in convergence guarantees for reinforcement learning with uncertainties and path planning algorithms for autonomous vehicles.
\end{rema}

\section*{Conclusions}
This paper presents a method for recovering the feasibility and Lipschitz regularity of the value function for control problems with time-dependent state constraints and infinite horizon discount factor. These results are essential for addressing optimal synthesis and weak solutions to the Hamilton-Jacobi-Bellman equation. We establish sufficient conditions on the constraint set to ensure feasibility and obtain estimates on the neighboring set of feasible trajectories, based on recent viability results. An important contribution of this paper is the demonstration of the equivalence between the master and relaxed infinite horizon problems. Additionally, we prove that the value function approaches zero at infinity for all feasible sets and large discount factors.

\section*{Conflict of Interest}
The author declare no conflicts of interest in this paper.

\bibliographystyle{plain}
\bibliography{BIBLIO_UPMC_IMJ_nov_2017_new}

\end{document}